\documentclass{amsart}
\usepackage{graphicx}
\usepackage{amscd}
\usepackage{color}
\usepackage{mathabx}
\newtheorem{theorem}{Theorem}[section]
\newtheorem{lemma}[theorem]{Lemma}
\newtheorem{claim}[theorem]{Claim}
\newtheorem{problem}[theorem]{Problem}

\newtheorem{proposition}[theorem]{Proposition}

\theoremstyle{definition}
\newtheorem{definition}[theorem]{Definition}

\begin{document}

\title[A partial order on multibranched surfaces]{A partial order on multibranched surfaces in 3-manifolds}

\author{Makoto Ozawa}
\address{Department of Natural Sciences, Faculty of Arts and Sciences, Komazawa University, 1-23-1 Komazawa, Setagaya-ku, Tokyo, 154-8525, Japan}
\email{w3c@komazawa-u.ac.jp}
\thanks{The author is partially supported by Grant-in-Aid for Scientific Research (C) (No. 17K05262) and (B) (No. 16H03928), The Ministry of Education, Culture, Sports, Science and Technology, Japan}



\begin{abstract}
In this paper, we introduce a partial order on neighborhood equivalence classes of maximally spread essential multibranched surfaces embedded in a 3-manifold.
We show that if a maximally spread essential multibranched surface is atoroidal and acylindrical, then its equivalence class is minimal with respect to the partial order.
\end{abstract}

\maketitle


\section{Introduction}

\subsection{Definition of multibranched surfaces}

Let $\Bbb{R}^2_+$ be the closed upper half-plane $\{(x_1,x_2)\in \Bbb{R}^2 \mid x_2\ge 0\}$.
The {\em multibranched Euclidean plane}, denoted by $S_i$ $(i\ge 1)$, is the quotient space obtained from $i$ copies of $\Bbb{R}^2_+$ by identifying with their boundaries $\partial \Bbb{R}^2_+=\{(x_1,x_2)\in\Bbb{R}^2\mid x_2=0\}$ via the identity map.
See Figure \ref{model} for the multibranched Euclidean plane $S_5$.

\begin{figure}[htbp]
	\begin{center}
	\includegraphics[trim=0mm 0mm 0mm 0mm, width=.4\linewidth]{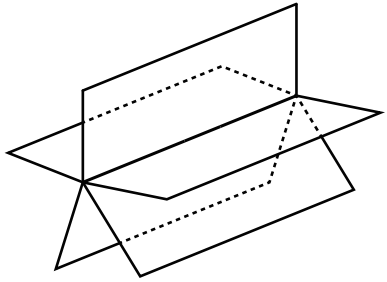}
	\end{center}
	\caption{The multibranched Euclidean plane $S_5$}
	\label{model}
\end{figure}

Informally, a multibranched surface is a space that is modeled on multibranched Euclidean plane.

\begin{definition}
A second countable Hausdorff space $X$ is called a {\em multibranched surface} if $X$ contains a disjoint union of simple closed curves $l_1,\ldots, l_n$ satisfying the following:
\begin{itemize}
\item For each point $x\in l_1\cup \cdots \cup l_n$, there exist an open neighborhood $U$ of $x$ and a positive integer $i$ such that $U$ is homeomorphic to $S_i$.
\item For each point $x\in X-(l_1\cup\cdots\cup l_n)$, there exists an open neighborhood $U$ of $x$ such that $U$ is homeomorphic to $\Bbb{R}^2$.
\end{itemize}
\end{definition}

We call each simple closed curve $l_i$ a {\em branch} and put $B_X=l_1\cup \cdots \cup l_n$.
For a branch $l_i$, we define the {\em degree} of $l_i$, which is denoted by $deg(l_i)$ as the positive integer $k$ if a point $x\in l_i$ has an open neighborhood $U$ such that $U$ is homeomorphic to $S_k$.
We call each component $e_j$ of $X-(l_1\cup\cdots\cup l_n)$ a {\em sector}.
For each sector $e_j$, $\bar{e}_j$ denotes a compact surface with $int\bar{e}_j=e_j$, where $int$ denotes the interior.
Let $E_X$ denote a disjoint union of all compact surfaces $\bar{e}_j$.

\subsection{Object of multibranched surfaces}

Multibranched surfaces naturally arise in several areas:
\begin{itemize}
\item Poly-continuous patterns --- a mathematical model of microphase-separated structures made by block copolymers (\cite{HCO}).
\item 2-stratifolds --- as spines of closed 3-manifolds (\cite{GGH}).
\item Trisections --- as a 4-dimensional analogue of Heegaard splittings (\cite{GK}). (In this case, we need to generalize the ``multibranched 3-dimensional Euclidean space" and define a ``multibranched 3-manifold" similarly.)
\item Links with non-meridional essential surfaces (\cite{EO}).
\end{itemize}

In \cite{MO}, it was shown that every multibranched surface is embeddable into the 4-dimensional Euclidean space $\Bbb{R}^4$, and a necessary and sufficient condition for a multibranched surface to be embeddable into some closed orientable 3-dimensional manifold was given.
In this paper, we consider multibranched surfaces embedded in a closed orientable 3-manifold.
Our goal is the following problem.

\begin{problem}\label{goal}
For a given closed orientable 3-manifold $M$, classify all multibranched surfaces embedded in $M$.
\end{problem}

\subsection{Essential multibranched surfaces}

To approach the goal, first we reduce the object into ``incompressible" one as below.

Let $M$ be a closed orientable 3-manifold and $X$ be a compact multibranched surface embedded in $M$.
We say that an orientable sector $e_j$ is {\em compressible} if there exists a disk $D$ embedded in $M$ such that $D\cap X=D\cap e_j=\partial D$ and $\partial D$ is essential in $e_j$ (i.e. there exists no disk $D'$ in $e_j$ such that $\partial D'=\partial D$).
We call such a disk $D$ a {\em compressing disk}.
(For a non-orientable sector $e_j$, we take a twisted $I$-bundle $e_j\tilde{\times} I$, and $e_j$ is said to be {\em compressible} if $e_j\tilde{\times} \partial I$ is compressible. We note that a compressing disk is not contained in $e_j\tilde{\times} I$.)
If a multibranched surface $X$ is compressible, then we do the following operation which is called a {\em compression} of $X$ along a compressing disk $D$ (Figure \ref{compression}).
\begin{enumerate}
\item Remove $int N(\partial D; e_j)$ from $e_j$, where $N(A;B)$ denotes the regular neighborhood of $A$ in $B$.
\item Paste two copies of $D$ along the boundary of $e_j-int N(\partial D; e_j)$.
\end{enumerate}
Then, we obtain another multibranched surface $X'$ as a result of this compression.
\begin{figure}[htbp]
	\begin{center}
	\includegraphics[trim=0mm 0mm 0mm 0mm, width=.6\linewidth]{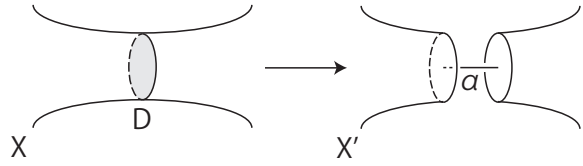}
	\end{center}
	\caption{A compression of $X$ along $D$}
	\label{compression}
\end{figure}

Conversely, we can recover the original multibranched surface $X$ from $X'$ by a {\em tubing} of $X'$ along a ``dual" arc $\alpha$, where $\alpha$ is an arc $\{x\}\times I\subset D\times I\cong N(D)$.
Since $X$ is compact, if we do compressions as much as possible, then we obtain a multibranched surface which is not compressible.
We say that a multibranched surface $X$ is {\em incompressible} in $M$ if it is not compressible.
In this sense, an incompressible multibranched surface can be regarded as a ``basis" of all multibranched surface embedded in $M$ since any multibranched surface can be obtained from incompressible multibranched surfaces by tubings.
According to the above, we restrict multibranched surfaces to be incompressible.

When we consider incompressible surfaces embedded in 3-manifolds, we usually suppose that a surface is boundary-incompressible and not boundary-parallel.
We consider slightly more strong condition on multibranched surfaces as follows.
Let $X$ be an incompressible multibranched surface embedded in a closed orientable 3-manifold $M$.
We say that a sector $e_j$ is {\em excess} if it is boundary-parallel in $M-int N(X-e_j)$.
(For a non-orientable sector $e_j$, we take a twisted $I$-bundle $e_j\tilde{\times} I$, and $e_j$ is said to be {\em excess} if $e_j\tilde{\times} \partial I$ is excess.)
A multibranched surface $X$ is said to be {\em efficient} if every sector is not excess.
We note that inefficient multibranched surfaces can be obtained from efficient multibranched surfaces.

In this paper, we say that a multibranched surface embedded in a closed orientable 3-manifold is {\em essential} if it is incompressible and efficient.
Hereafter, we consider only essential multibranched surfaces.







\subsection{Degree and wrapping number}

Any multibranched surface $X$ can be constructed from a closed 1-manifold $B_X=l_1\cup\cdots\cup l_m$ and a compact surfaces with boundary $E_X=\bar{e}_1\cup \cdots\cup \bar{e}_n$ by identifying via a covering map $f:\partial E_X\to B_X$.
The {\em degree} of a branch $l_i$ is defined as $deg(l_i)=d$ if $f|_{f^{-1}(l_i)}:f^{-1}(l_i)\to l_i$ is a $d$-fold covering.
For each component $C$ of $\partial E_X$, the {\em wrapping number} of $C$ is defined as $wrap(C)=w$ if $f|_{C}$ is a $w$-fold covering.
See Figure \ref{degree}.

\begin{figure}[htbp]
	\begin{center}
	\includegraphics[trim=0mm 0mm 0mm 0mm, width=.6\linewidth]{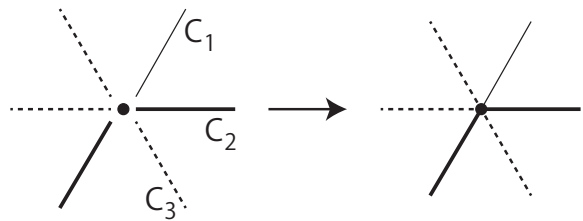}
	\end{center}
	\caption{$deg(l_i)=6$,  $wrap(C_i)=1,\ 2,\ 3$}
	\label{degree}
\end{figure}

We say that a branch $l_i$ is {\em normal} if for each component $C$ of $f^{-1}(l_i)$, $wrap(C)=1$, namely, the number of components of $f^{-1}(l_i)$ is equal to $deg(l_i)$.
We say that a branch $l_i$ is {\em pure} if $f^{-1}(l_i)$ consists of a single component $C$, namely, $wrap(C)=deg(l_i)$.

\subsection{IX-moves and XI-moves}

There are local moves on multibranched surfaces which are analogues to the edge contraction in graph theory.
Suppose that there is an open annulus sector $e_j$ and put $A=\bar{e}_j$ and $\partial A=a_1\cup a_2$.
We assume that at least one boundary component of $A$, say $a_1$, has the wrapping number 1.
We say that $A$ is {\em normal} if both boundary components of $A$ have the wrapping number 1, and that $A$ is {\em quasi-normal} if one boundary component of $A$ has the wrapping number 1 and another boundary component of $A$ has the wrapping number greater than 1.
Suppose that there is an open M\"{o}bius band sector $e_j$ and put $M=\bar{e}_j$.
We say that $M$ is {\em normal} if the boundary of $M$ has the wrapping number 1.
An {\em IX-move} along $\bar{e_j}$ is a deformation retraction of $\bar{e_j}$ onto the core circle (resp. the unnormal branch) if $e_j$ is either a normal annulus sector or a normal M\"{o}bius band sector (resp. a quasi-normal annulus sector) (\cite{IKOS}).
See Figures \ref{IX1}, \ref{IX2} and \ref{IX3} for an IX-move along a normal annulus sector, a quasi-normal annulus sector and a normal M\"{o}bius band sector respectively.
An {\em XI-move} is a reverse operation of an IX-move.

\begin{figure}[htbp]
	\begin{center}
	\includegraphics[trim=0mm 0mm 0mm 0mm, width=.6\linewidth]{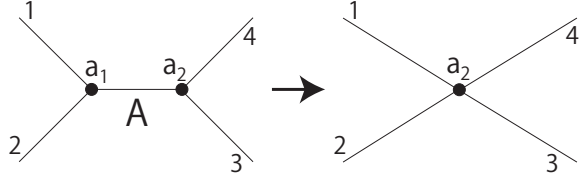}
	\end{center}
	\caption{IX-move along $A$, where $wrap(a_i)=1$ for $i=1,2$.}
	\label{IX1}
\end{figure}

\begin{figure}[htbp]
	\begin{center}
	\includegraphics[trim=0mm 0mm 0mm 0mm, width=.6\linewidth]{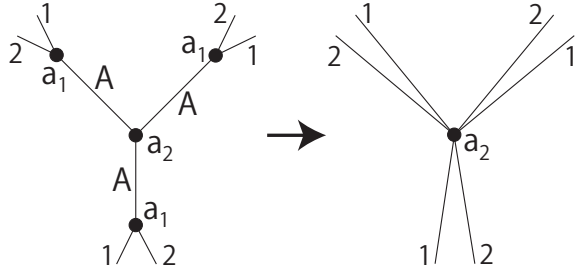}
	\end{center}
	\caption{IX-move along $A$, where $wrap(a_1)=1$ and $wrap(a_2)=3$}
	\label{IX2}
\end{figure}

\begin{figure}[htbp]
	\begin{center}
	\includegraphics[trim=0mm 0mm 0mm 0mm, width=.6\linewidth]{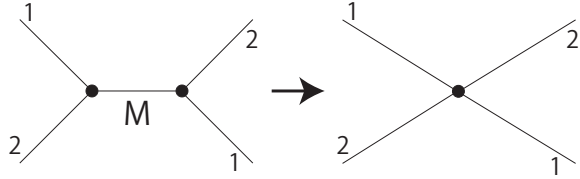}
	\end{center}
	\caption{IX-move along a normal M\"{o}bius band sector}
	\label{IX3}
\end{figure}

\subsection{Neighborhood equivalence of multibranched surfaces}

Let $X, X'$ be multibranched surfaces embedded in a closed orientable 3-manifold $M$.
Suppose that $X'$ is obtained from $X$ by a finite sequence of IX-moves, XI-moves.
Then the regular neighborhood $N(X)$ of $X$ is isotopic to the regular neighborhood $N(X')$ of $X'$.
The following theorem states that the converse holds.

\begin{theorem}[\cite{IKOS}]\label{neighborhood}
Let $X, X'$ be multibranched surfaces embedded in a closed orientable 3-manifold $M$ which have no open disk sector and no branch of degree 1 or 2.
If $N(X)$ is isotopic to $N(X')$ in $M$, then $X$ is transformed into $X'$ by a finite sequence of IX-moves, XI-moves and isotopies in $M$.
\end{theorem}

We say that a multibranched surface $X$ is a {\em tribranched surface} if for each branch $l_i$ of $X$, $deg(l_i)= 3$.
As stated in \cite{IKOS}, the set of tribranched surfaces in a given 3-manifold $M$ is ``generic'', that is, it forms an open and dense subset in the space of all multibranched surfaces in a suitable sense.
 
We say that a branch is {\em non-spreadable} if it is normal and tribranched, or pure, otherwise we say that it is {\em spreadable}.
Note that each spreadable branch of $X$ admits an XI-move.
By applying XI-moves to $X$ maximally, we get a multibranched surface without spreadable branches.
We call such a multibranched surface a {\em maximally spread surface}.
For a maximally spread surface, an IX-move followed by an XI-move is called an {\it IH-move}. 

Theorem \ref{neighborhood} is proved by showing the next theorem.

\begin{theorem}[\cite{IKOS}]\label{IH}
Let $X$ and $X'$ be maximally spread surfaces in $N$ such that $N$ is a regular neighborhood of each of $X$ and $X'$.
Then $X$ is transformed into $X'$ by a finite sequence of IH-moves and isotopies.
\end{theorem}

\subsection{Our objects $\mathcal{X}$}

Hence, to consider Problem \ref{goal}, we restrict multibranched surfaces to the set $\mathcal{X}$ of all connected compact multibranched surfaces $X$ embedded in a closed orientable 3-manifold $M$ satisfying the following conditions:

\begin{enumerate}
\item $X$ is maximally spread.
\item $X$ is essential in $M$.
\item $X$ has no open disk sector.
\item $X$ has no branch of degree 1 or 2.
\end{enumerate}

Under the influence of Theorem \ref{IH}, we define an equivalence relation on $\mathcal{X}$ as follows.
Two multibranched surfaces $X$ and $X'$ in $\mathcal{X}$ are {\em equivalent}, denoted by $X\sim X'$, if $X$ is transformed into $X'$ by a finite sequence of IH-moves.
This equivalence relation is closed in $\mathcal{X}$ as follows.

\begin{proposition}\label{closed}
Let $X, X'$ be multibranched surfaces embedded in a closed orientable 3-manifold $M$.
If $X\in \mathcal{X}$ and $X\sim X'$, then $X'\in \mathcal{X}$.
\end{proposition}

\subsection{A partial order on multibranched surfaces}

We define a binary relation $\le$ over $\mathcal{X}$ as follows.

\begin{definition}\label{relation}
For $X,\, Y\in \mathcal{X}$, we denote $X \le Y$ if 
\begin{enumerate}
\item there exists an isotopy of $Y$ in $M$ so that $Y\subset N(X)$ and $B_Y\subset N(B_X)$, and
\item there exists no essential annulus in $N(X)-Y$.
\end{enumerate}
\end{definition}
The second condition (2) says that for any annulus $A$ properly embedded in $N(X)-Y$, either $A$ is compressible in $N(X)-Y$ or $A$ is isotopic to a subannulus of $\partial N(X)$ in $N(X)-Y$.

For equivalence classes $[X], [Y]\in \mathcal{X}/\sim$, we define a binary relation $\preceq$ over $\mathcal{X}/\sim$ so that $[X]\preceq [Y]$ if $X\le Y$.

\begin{proposition}\label{well-defined}
The relation $\preceq$ is well-defined on $\mathcal{X}/\sim$.
\end{proposition}

The next is a main theorem in this paper.

\begin{theorem}\label{partial}
$(\mathcal{X}/\sim; \preceq)$ is a partially ordered set.
\end{theorem}


Since the partial order is defined on the quotient set of $\mathcal{X}$ by $\sim$, Theorem \ref{partial} also provides a partial order on the set of all compact orientable 3-manifolds whose spines are multibranched surfaces belong to $\mathcal{X}$.
Then it is natually posed whether it can be generalized to all compact orientable 3-manifolds.

\begin{problem}
Is it possible to define a partial order on the set of all compact orientable 3-manifolds which extends the partially ordered set in Theorem \ref{partial}?
\end{problem}

\subsection{A sufficient condition to be minimal}

Recall that $B_X$ denotes the disjoint union of all branches of a multibranched surface $X$, and  $E_X$ denotes a disjoint union of all compact surfaces $\bar{e}_j$, where $e_j$ is a sector of $X$.

Suppose that $X$ is embedded in a closed orientable 3-manifold $M$.
We say that $B_X$ is {\em toroidal} if there exists an essential torus $T$ in the exterior $E(B_X)$ of $B_X$ in $M$, that is, $T$ is incompressible in $E(B_X)$ and $T$ is not parallel to a torus in $\partial E(B_X)$.
We say that $E_X$ is {\em cylindrical} if there exists an essential annulus $A$ with $A\cap X =A\cap E_X=\partial A$, that is, $A$ is incompressible and $A$ is parallel to neither an annulus in $E_X$ nor an annulus in $\partial E(B_X)$.

\begin{theorem}\label{sufficient}
For equivalence classes $[X], [Y]\in \mathcal{X}/\sim$, 
if $[X]\preceq [Y]$ and $[X]\ne [Y]$, then either $B_Y$ is toroidal or $E_Y$ is cylindrical.
\end{theorem}

Theorem \ref{sufficient} provides a sufficient condition for an equivalent class $[X]\in \mathcal{X}/\sim$ to be minimal with respect to the partial order of $(\mathcal{X}/\sim; \preceq)$, that is, if $B_X$ is atoroidal and $E_X$ is acylindrical, then $[X]$ is minimal.

\section{Preliminaries}

Recall that $B_X$ denotes the disjoint union of all branches of a multibranched surface $X$, and  $E_X$ denotes a disjoint union of all compact surfaces $\bar{e}_j$, where $e_j$ is a sector of $X$.
Since a multibranched surface $X$ is a union of $B_X$ and $E_X$, we can regard $N(X)$ as a union of $N(B_X)$ and $N(E_X)$ with $N(B_X)\approx B_X \times D^2$, $N(E_X)\approx E_X\times I$ (for a non-orientable sector, we take a twisted $I$-bundle), where $D^2$ denotes the closed disk and $I=[-1,1]$, $A \approx B$ means that $A$ is homeomorphic to $B$.
The intersection $N(B_X)\cap N(E_X)$ consists of annuli $A_X$, which are called the {\it characteristic annulus system}.

\begin{figure}[htbp]
	\begin{center}
	\includegraphics[trim=0mm 0mm 0mm 0mm, width=.4\linewidth]{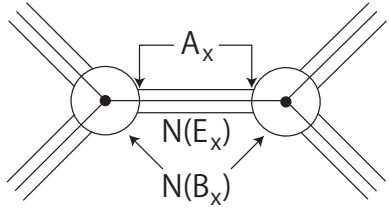}
	\end{center}
	\caption{The characteristic annulus system}
	\label{characteristic}
\end{figure}

\begin{proposition}\label{annulus}
Let $X\in \mathcal{X}$.
The characteristic annulus system $A_X$ is essential in $N(X)$, and $N(X)$ is irreducible and boundary-irreducible.
\end{proposition}

\begin{proof}
Suppose that $A_X$ is inessential in $N(X)$.
Then a component $A$ in $A_X$ is compressible or boundary-parallel in $N(X)$.
Let $D$ be a compressing disk for $A$.
Since the core of $A$ is not isotopic to a meridian of $N(B_X)$, it is not contractible in $N(B_X)$.
Hence $D$ is contained in $N(E_X)$.
But this implies that the sector $e_j$ with $D\subset N(\bar{e}_j)$ is a disk.
It contradicts the condition (3) of $\mathcal{X}$.
Next suppose that $A$ is boundary-parallel in $N(X)$.
This implies that the branch $l_i$ with $A\subset \partial N(l_i)$ has degree 1.
It also contradicts the condition (4) of $\mathcal{X}$.

Suppose that $N(X)$ is reducible.
Then there exists an essential 2-sphere $S$ in $N(X)$.
If $S\cap A_X=\emptyset$, then $S$ is contained in $N(B_X)$ or $N(E_X)$.
Since both $N(B_X)$ and $N(E_X)$ are irreducible, $S$ bounds a 3-ball in $N(X)$.
This is a contradiction.
Otherwise, we isotope $S$ so that $S$ intersects $A_X$ in loops, and assume that the number of components $|S\cap A_X|$ is minimal among all essential 2-spheres.
If there is a loop of $S\cap A_X$ which is inessential in $A_X$, then by cutting and pasting $S$ along an innermost disk in $A_X$, we obtain another essential 2-sphere $S'$ with $|S'\cap A_X|<|S\cap A_X|$.
This contradicts the minimality of $|S\cap A_X|$.
Hence, all loops of $S\cap A_X$ are essential in $A_X$.
Let $\delta$ be an innermost disk in $S$ with respect to $S\cap A_X$.
If $\delta$ is contained in $N(B_X)$, then we have a contradiction since $\partial \delta$ is an essential, non-meridional loop in $\partial N(B_X)$.
If $\delta$ is contained in $N(E_X)$, then this implies the sector is an open disk.
This contradicts that $X$ has no open disk sector, and contradicts the condition (3) of $\mathcal{X}$.

Next, suppose that $N(X)$ is boundary-reducible.
Then there exists an essential disk $D$ in $N(X)$.
If $D\cap A_X=\emptyset$, then $D$ is contained in $N(B_X)$ or $N(E_X)$.
Since both $\partial N(B_X)-A_X$ and $\partial N(E_X)-A_X$ are incompressible, $\partial D$ bounds a disk in $\partial N(X)$.
This contradicts the condition of $\mathcal{X}$.
Otherwise, we isotope $D$ so that $D$ intersects $A_X$ in loops and arcs, and assume that $|D\cap A_X|$ is minimal among all essential disks in $N(X)$.
If there is a loop or arc of $D\cap A_X$ which is inessential in $A_X$, then by cutting and pasting $D$ along an innermost disk or outermost disk in $A_X$, we obtain another essential disk $D'$ with $|D'\cap A_X|<|D\cap A_X|$.
This contradicts the minimality of $|D\cap A_X|$.
Hence, all loops and arcs of $D\cap A_X$ are essential in $A_X$.
Let $\delta$ be an innermost disk in $D$ with respect to $D\cap A_X$.
If $\delta$ is contained in $N(B_X)$, then we have a contradiction since $\partial \delta$ is an essential, non-meridional loop in $\partial N(B_X)$.
If $\delta$ is contained in $N(E_X)$, then this implies the sector is an open disk.
This contradicts that $X$ has no open disk sector.
Let $\delta$ be an outermost disk in $D$ with respect to $D\cap A_X$.
If $\delta$ is contained in $N(B_X)$, then it is a meridian disk of $N(B_X)$ which intersects the characteristic annulus in a single arc.
This implies that the degree of the branch is equal to 1, and contradicts the condition (4) of $\mathcal{X}$.
If $\delta$ is contained in $N(E_X)$, then by cutting $A_X$ along $\partial \delta\cap A_X$ and pasting two parallel copies of $\delta$, we obtain a compressing disk for $\partial N(E_X)-A_X$, or a component of $N(E_X)$ is a solid torus with $\delta$ as a meridian disk.
In both cases, we have a contradiction.
\end{proof}

\begin{proof}[Proof of Proposition \ref{closed}]
It is sufficient to show that if $X'$ is obtained from $X$ by a single IH-move along a normal or quasi-normal annulus or normal M\"{o}bius band sector $A$, then $X'\in\mathcal{X}$.

(1) Suppose that $X'$ has a spreadable branch.
Then $X$ has also a spreadable branch and this contradicts that $X$ is maximally spread.

(2) Suppose that $X'$ is compressible.
Then there exists a compressing disk $D$ for some sector $e_j$ of $X'$.
If $e_j$ is not $A$, then $D$ is also a compressing disk for some sector of $X$ corresponding to $e_j$.
This contradicts that $X$ is incompressible.
Otherwise, $D$ is a compressing disk for $A$.
In this case, we can isotope $D$ off $A$ so that it is a compressing disk for another sector which shares same branch with $A$.
Thus it is in the previous case and we have a contradiction.

Suppose that $X'$ is inefficient and a sector $e_j$ of $X'$ is excess.
If $e_j$ is not $A$, then $e_j$ is also excess in $X$ since $N(X'-e_j)$ is isotopic to $N(X-e_j)$.
This contradicts that $X$ is efficient.
Otherwise, $A$ is boundary-parallel in $M-int N(X'-A)$.
This implies that there is a solid torus region $V$ in $M-X'$ and $A$ goes around a longitude of $V$ exactly once.
There is also a solid torus region $V$ in $M-X$ which corresponds to $V$, and each sector goes around a longitude of $V$ exactly once.
It follows that an annulus sector in $\partial V$ is excess. See Figure \ref{excess}.
This contradicts that $X$ is efficient.
We remark that in the latter case, $A$ is not a quasi-normal annulus.

\begin{figure}[htbp]
	\begin{center}
	\includegraphics[trim=0mm 0mm 0mm 0mm, width=.45\linewidth]{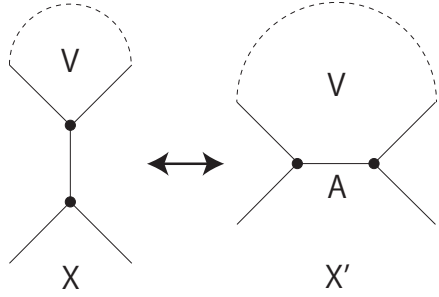}
	\end{center}
	\caption{The solid torus region $V$}
	\label{excess}
\end{figure}

(3) Suppose that $X'$ has an open disk sector.
Then $X$ has also an open disk sector, and this is a contradiction.

(4) Suppose that $X'$ has a branch of degree 1 or 2.
Then $X$ has also a branch of degree 1 or 2, and this is a contradiction.
\end{proof}

The next lemma is used several times.

\begin{lemma}[{\cite[Proposition 3.1]{W}}]\label{product}
Let $F$ be an orientable surface which is not the 2-sphere and $I$ be the closed interval $[0,1]$.
If $G$ is an incompressible surface embedded in $F\times I$ with $\partial G\subset F\times \{0\}$, then $F$ is parallel to a surface in $F\times \{0\}$.
\end{lemma}

\section{Proof of Theorem \ref{partial}}

We first show the well-definedness.

\begin{proof}[Proof of Proposition \ref{well-defined}]
Suppose that $[X]\preceq [Y]$, then $X\le Y$.
Let $X'\in[X]$ and $Y'\in[Y]$, then $X$ and $Y$ is transformed into $X'$ and $Y'$ respectively by a finite sequence of IH-moves.
We need to show that $X'\le Y'$.

Since $Y\subset N(X)$, $Y$ can be transformed into $Y'$ in $N(X)$ by a finite sequence of IH-moves.
By Theorem \ref{neighborhood}, $N(X)$ is isotopic to $N(X')$ in $M$.
Hence, we have $Y'\subset N(X')$.
We should check that the condition $B_Y\subset N(B_X)$ can be kept during these deformations.
Let $X=X_1,\ldots, X_n=X'$ and $Y=Y_1,\ldots,Y_m=Y'$ be the finite sequence of IH-moves, and $(X_1,Y_1),\ldots,(X_{m+n},Y_{m+n})$ be a sequence of pairs which includes those sequences such that exactly one of the following cases holds.
\begin{description}
\item[Case 1] $X_{i+1}$ is obtained from $X_i$ by an IH-move, and $Y_{i+1}=Y_i$.
\item[Case 2] $Y_{i+1}$ is obtained from $Y_i$ by an IH-move, and $X_{i+1}=X_i$.
\end{description}

We need the next lemma.

\begin{lemma}\label{inside}
Suppose that two multibranched surfaces $Z$ and $Z'$ in $\mathcal{X}$ are related by a single IH-move.
Then $Z'$ can be isotoped locally so that $Z'\subset N(Z)$ and $B_{Z'}\subset N(B_Z)$.
\end{lemma}

\begin{proof}
Suppose that $Z'$ is obtained from $Z$ by a single IH-move along a normal annulus sector, a quasi-normal annulus sector or a normal M\"{o}bius band sector $A$.
It can be observed that $Z'$ is a result by sliding sectors of $Z$ along $A$.
See Figures \ref{IH1} and \ref{IH2} for an IH-move along a normal annulus sector and a quasi-normal annulus sector respectively.
The case of an IH-move along a normal M\"{o}bius band sector is similar.
\begin{figure}[htbp]
	\begin{center}
	\includegraphics[trim=0mm 0mm 0mm 0mm, width=.4\linewidth]{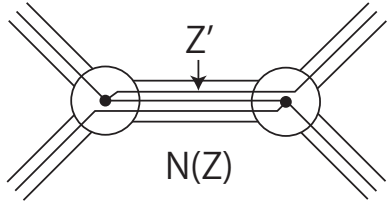}
	\end{center}
	\caption{IH-move along a normal annulus sector}
	\label{IH1}
\end{figure}
\begin{figure}[htbp]
	\begin{center}
	\includegraphics[trim=0mm 0mm 0mm 0mm, width=.5\linewidth]{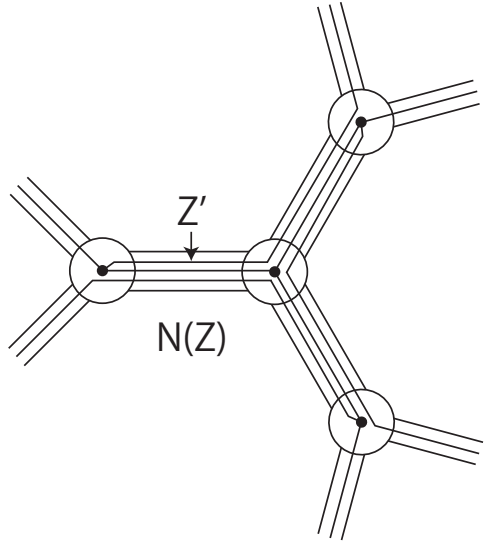}
	\end{center}
	\caption{IH-move along a quasi-normal annulus sector}
	\label{IH2}
\end{figure}
\end{proof}

We continue the proof of Proposition \ref{well-defined}.

In Case 1, by Lemma \ref{inside}, $X_i$ can be isotoped locally so that $X_i\subset N(X_{i+1})$ and $B_{X_i}\subset N(B_{X_{i+1}})$.
Since $Y_i\subset N(X_i)$ and $B_{Y_i}\subset N(B_{X_i})$, it follows that $Y_i\subset N(X_{i+1})$ and $B_{Y_i}\subset N(B_{X_{i+1}})$.

In Case 2, by Lemma \ref{inside}, $Y_{i+1}$ can be isotoped locally so that $Y_{i+1}\subset N(Y_i)$ and $B_{Y_{i+1}}\subset N(B_{Y_i})$.
Since $Y_i\subset N(X_i)$ and $B_{Y_i}\subset N(B_{X_i})$, it follows that $Y_{i+1}\subset N(X_i)$ and $B_{Y_{i+1}}\subset N(B_{X_i})$.

Therefore, the first condition (1) of Definition \ref{relation} is satisfied.
The second condition (2) of Definition \ref{relation} holds since the pair $(N(X), N(Y))$ is isotopic to $(N(X'), N(Y'))$.
\end{proof}

Next we prepare some lemmas to prove Theorem \ref{partial}.
The following lemma is important since it arranges ordered multibranched surfaces.

\begin{lemma}[Standard form]\label{standard}
Suppose that for $X,\, Y\in \mathcal{X}$, $X \le Y$.
Then $Y$ can be isotoped in $M$ so that the following conditions hold.
\begin{enumerate}
\item $Y\subset N(X)$ and $B_Y\subset N(B_X)$.
\item There exists no essential annulus in $N(X)-Y$.
\item $Y\cap N(E_X)$ consists of surfaces with a form $E_X\times \{x\}$.
\item $Y\cap N(B_X)$ is incompressible in $N(B_X)$ and no component of $Y\cap N(B_X)$ is parallel into $A_X$.
\end{enumerate}
\end{lemma}

\begin{proof}
Suppose that for $X,\, Y\in \mathcal{X}$, $X \le Y$.
By Definition \ref{relation}, $Y$ can be isotoped in $M$ so that the conditions (1) and (2) are satisfied.
In the following, we isotope $Y$ in $N(X)$ under holding (1) and (2) so that the conditions (3) and (4) are satisfied.

By the incompressibility of $Y$ and Proposition \ref{annulus}, $Y$ can be isotoped in $N(X)$ so that $Y$ intersects the characteristic annulus system $A_X$ only in loops which are essential in both $Y$ and $A_X$.
Then, $Y\cap N(E_X)$ and $Y\cap N(B_X)$ is incompressible in $N(E_X)$ and $N(B_X)$ respectively.
We assume that $|Y\cap A_X|$ is minimal.
Then, there is no component of $Y\cap N(E_X)$ or $Y\cap E(B_X)$ parallel into $A_X$.
By Lemma \ref{product}, each component of $Y\cap N(E_X)$ is parallel to $E_X$.
Therefore, the conditions (3) and (4) are satisfied.
\end{proof}

We say that $Y$ is in a {\em standard form} if the conditions (1)-(4) of Lemma \ref{standard} are satisfied.
The next lemma is fundamental since it is useful to order multibranched surfaces by the Euler characteristic of sectors.

\begin{lemma}\label{euler}
Suppose that for $X,\, Y\in \mathcal{X}$, $X \le Y$.
Then $\chi(E_Y)\le \chi(E_X)$.
Moreover $\chi(E_Y)= \chi(E_X)$ if and only if 
\begin{enumerate}
\item For any sector $F\in E_X$ except for annulus or M\"{o}bius band sectors, $E_Y\cap N(F)$ consists of a single surface.
\item $E_Y\cap N(B_X)$ consists of annuli.
\end{enumerate}
\end{lemma}

\begin{proof}
By Lemma \ref{standard}, we may assume that $Y$ is isotoped to be in a standard form.

Consider the following diagram, where $p$ is a covering map by Lemma \ref{standard} (3).
$$
E_Y\supset E_Y\cap N(E_X) \xrightarrow{p} E_X
$$
From this diagram, we have the following inequality.
$$
\chi(E_Y)\le \chi(E_Y\cap N(E_X)) \le \chi(E_X)
$$
For the first inequality, we have the next equality more precisely.
$$
\chi(E_Y)= \chi(E_Y\cap N(E_X))+\chi(E_Y\cap N(B_X)),
$$
For the second inequality, $\chi(E_Y\cap N(E_X)) = \chi(E_X)$ if and only if $p$ is injective for a sector $F\in E_X$ except for an annulus or a M\"{o}bius band sectors.

It follows from these observations that $\chi(E_Y)= \chi(E_X)$ if and only if $\chi(E_Y\cap N(B_X))=0$, namely $E_Y\cap N(B_X)$ consists of annuli or M\"{o}bius bands, and for any sector $F\in E_X$ except for annulus or M\"{o}bius band sectors, $E_Y\cap N(F)$ consists of a single surface.

Finally we show that $E_Y\cap N(B_X)$ does not have a M\"{o}bius band component.
Suppose that there is a branch $l\in B_X$ such that $E_Y\cap N(l)$ contains a M\"{o}bius band.
Then for each component $C$ of $f^{-1}(l)$, $wrap(C)=2$, where $f:\partial E_X\to B_X$ is a covering map.
Since $X\in\mathcal{X}$, $deg(l)\ge 3$.
Hence $deg(l)$ is an even integer greater than 3, and $l$ is unnormal and not pure.
This implies that $l$ is spreadable and contradicts that $X$ is maximally spread.
\end{proof}

\begin{proof}[Proof of Theorem \ref{partial}]

We need to show that:

\begin{description}
\item[Reflexivity] For any $X\in \mathcal{X}$, $[X]\preceq [X]$.
\item[Antisymmetry] For $X,Y\in \mathcal{X}$, if $[X]\preceq [Y]$ and $[Y]\preceq [X]$, then $[X]=[Y]$.
\item[Transitivity] For $X,Y,Z\in \mathcal{X}$, if $[X]\preceq [Y]$ and $[Y]\preceq [Z]$, then $[X]\preceq [Z]$.
\end{description}

\medskip
Reflexivity: 
For any $X\in \mathcal{X}$, it holds that $X\subset N(X)$ and $B_X\subset N(B_X)$ without an isotopy of $X$ in $M$.
Suppose that there exists an essential annulus $A$ in $N(X)-X$.
Since $N(X)-X$ is homeomorphic to a product $\partial N(X)\times (0,1]$, by Lemma \ref{product}, $A$ is parallel into $\partial N(X)$ in $N(X)-X$.
This is a contradiction.

\medskip
Transitivity: 
Suppose that for $X,Y,Z\in \mathcal{X}$, $[X]\preceq [Y]$ and $[Y]\preceq [Z]$.
Since $Y \le Z$, by Definition \ref{relation}, there exists an isotopy of $Z$ in $M$ so that $Z\subset N(Y)$ and $B_Z\subset N(B_Y)$, and there exists no essential annulus in $N(Y)-Z$.
By Lemma \ref{standard}, we may assume that $Z$ is in a standard form for $Y$.
Moreover, since $X \le Y$, there exists an isotopy of $N(Y)$ in $M$ so that $N(Y)\subset N(X)$ and $N(B_Y)\subset N(B_X)$, and there exists no essential annulus in $N(X)-N(Y)$.
By Lemma \ref{standard}, we may assume that $Y$ is in a standard form for $X$.
Hence $Z\subset N(X)$ and $B_Z\subset N(B_X)$, thus the condition (1) of Definition \ref{relation} is satisfied.
Next we show that there exists no essential annulus in $N(X)-Z$.
Suppose that there exists an essential annulus $A$ in $N(X)-Z$.
Since there exists no essential annulus in $N(X)-N(Y)$, $A$ must intersect $\partial N(Y)$.

\begin{claim}\label{arrange}
$A$ can be isotoped in $N(X)-Z$ so that $A\cap \partial N(Y)=A\cap (\partial N(B_Y)-A_Y)$.
Moreover, if we take $|A\cap (\partial N(B_Y)-A_Y)|$ minimal up to isotopy, $A\cap (\partial N(B_Y)-A_Y)$ consists of loops which are essential in both $A$ and $\partial N(B_Y)-A_Y$.
\end{claim}

\begin{proof}
For each component $A_i$ of the characteristic annulus system $A_Y$, let $A_{i1},\ldots,A_{in_i}$ be the annulus components of $A_i-Z$, where $A_{i1}$ and $A_{in_i}$ meet with $\partial N(Y)$.

\begin{figure}[htbp]
	\begin{center}
	\includegraphics[trim=0mm 0mm 0mm 0mm, width=.5\linewidth]{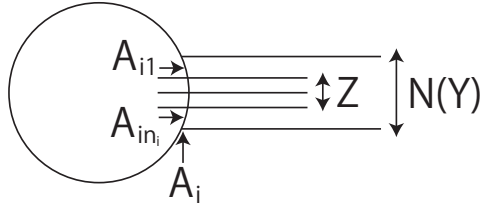}
	\end{center}
	\caption{Configuration of $N(Y)$, $Z$, $A_i$, $A_{i1},\ldots,A_{in_i}$}
	\label{characteristic}
\end{figure}

First we isotope $A$ in $N(X)-Z$ so that $A\cap (A_{i1}\cup A_{in_i})=\emptyset$ for each $A_i$.
Next by using the product structure of $N(E_Y)$, we isotope off $A$ from the components of $N(E_Y)-E_Z$ which meet with $A_{i1}$ and $A_{in_i}$ for all $i$.
Then it holds that $A\cap \partial N(Y)=A\cap (\partial N(B_Y)-A_Y)$.
Moreover, if we take $|A\cap (\partial N(B_Y)-A_Y)|$ minimal up to isotopy, then by the incompressibility of $A$ and $\partial N(B_Y)-A_Y$, $A\cap (\partial N(B_Y)-A_Y)$ consists of loops which are essential in both $A$ and $\partial N(B_Y)-A_Y$.
\end{proof}

By Claim \ref{arrange}, any annulus component of $A\cap N(Y)$ is essential in $N(Y)-Z$.
This contradicts that there exists no essential annulus in $N(Y)-Z$.

\medskip

Antisymmetry: 
Suppose that for $X,Y\in \mathcal{X}$, $[X]\preceq [Y]$ and $[Y]\preceq [X]$.
Then $X\le Y$ and $Y\le X'$, where $X'\in [X]$.
By the transitivity, $X\le X'$ and
$$
X'\subset N(Y)\subset N(X),\ B_{X'}\subset N(B_Y)\subset N(B_X).
$$
We may assume that $Y$ is in a standard form for $X$, and $X'$ is in a standard form for $Y$.
In the following, we show that $\partial N(X')$ is parallel to $\partial N(X)$ in $N(X)$.
It follows that by Lemma \ref{product}, $\partial N(Y)$ is also parallel to $\partial N(X)$ in $N(X)$, and hence $N(Y)\approx N(X)$.
Thus by Theorem \ref{IH}, $[Y]=[X]$.

Since $X'$ is obtained from $X$ by an IH-moves, $\chi(X')=\chi(X)$.
Hence the conditions (1) and (2) of Lemma \ref{euler} are satisfied for $X$ and $X'$.
By the standard form and (1) of Lemma \ref{euler}, $X'\cap N(E_X)$ consists of the following.
For each sector $F\in E_X$, $X'\cap N(F)$ is either:
\begin{enumerate}
\item a single surface parallel to $F$ if $F$ is not an annulus nor a M\"{o}bius band sector.
\item annuli with a form $F\times \{x\}\subset F\times I$ if $F$ is an annulus sector.
\item a M\"{o}bius band with a form $F\tilde{\times} \{0\}\subset F\tilde{\times} I$, or annuli with a form $F\tilde{\times} \{x\}\subset F\tilde{\times} I$ $(x\ne 0)$ if $F$ is a M\"{o}bius band sector, where $\tilde{\times}$ denotes the twisted $I$-bundle.
\end{enumerate}

For each branch $l\in B_X$, a solid torus $N(l)$ can be naturally regarded as a standard fibered torus, that is, an $S^1$-bundle over a disk $D$.
Let $p:N(l)\to D$ be the projection.
The core of $N(l)$ is a singular fiber if $l$ is pure, otherwise it is a regular fiber.
The core of each annulus of $A_X\cap \partial N(l)$, is a regular fiber.
By (2) of Lemma \ref{euler} and (4) of Lemma \ref{standard}, $X'\cap N(B_X)$ consists of the following.
For each branch $l\in B_X$, $X'\cap N(l)$ is either:
\begin{enumerate}
\item a branch $l'\in B_{X'}$ with annuli incident to $l'$, which projects to $C_d$ (or $C_1$ if $l'$ is a singular fiber in $N(l)$) in the base (orbit surface) of the fibered solid torus $N(l)$, where $C_d$ denote the cone over $d$ points.
\item an annulus which projects to an arc in the base of $N(l)$ which is not parallel to $p(A_X\cap \partial N(l))$.
\end{enumerate}

We consider the multibranched surface $X^*$ obtained from $X$ by replacing each sector which is not an annulus nor a M\"{o}bius band sector with $C_d\times S^1$.
The multibranched surface ${X'}^*$ is simultaneously obtained from $X'$.
Then $S=N(X^*)$ is a Seifert fibered 3-manifold with a base orbifold $B$.
Let $p:S\to B$ be the projection.
By the above observations, $X'$ consists of fibers in $S$ and $p({X'}^*)$ is a graph embedded in $B$.

\begin{claim}\label{product}
Let $R$ be the union of regions of $B-N(p({X'}^*))$ each of which contains a component of $\partial B$.
Then $R$ has a product structure $\partial B\times (0,1]$.
\end{claim}

\begin{proof}
Suppose that $R$ does not have a product structure $\partial B\times (0,1]$.
Then there exists an arc $\gamma$ properly embedded in $R$ with $\partial \gamma \subset \partial B$ such that $\gamma$ is not parallel to an arc in $\partial B$.
The fiber $p^{-1}(\gamma)$ is an annulus or a M\"{o}bius band properly embedded in $N(X)-X'$.
If $p^{-1}(\gamma)$ is an annulus, then it is an essential annulus in $N(X)-X'$.
This contradicts to the condition (2) in Definition \ref{relation} for $X\le X'$.
If $p^{-1}(\gamma)$ is a M\"{o}bius band, then we take $p^{-1}(\gamma)\tilde{\times} \partial I\subset p^{-1}(\gamma)\tilde{\times} I$ and it is an essential annulus in $N(X)-X'$.
This contradicts to the condition (2) in Definition \ref{relation} for $X\le X'$.
\end{proof}

By Claim \ref{product}, $\partial N(X')$ is parallel to $\partial N(X)$.
It follows that $\partial N(Y)$ is also parallel to $\partial N(X)$ in $N(X)$ and hence $[Y]=[X]$.
\end{proof}

\section{Proof of Theorem \ref{sufficient}}

\begin{proof}[Proof of Theorem \ref{sufficient}]

Suppose that $B_Y$ is atoroidal.
Then each branch of $Y$ is isotopic to the core of $N(l_i)$ for some branch $l_i$ of $X$, and each $N(l_i)$ contains at most one branch of $Y$.

Suppose that $E_Y$ is acylindrical.
If for a sector $F\in E_X$ which is not an annulus nor a pair of pants, $|E_Y\cap N(F)|\ge 2$ holds,
then there exists a vertical essential annulus between two components of $E_Y\cap N(F)$.
This is a contradiction.
Hence, for a sector $F\in E_X$ which is not an annulus nor a pair of pants, $|E_Y\cap N(F)|=1$.

Suppose that for a sector $F\in E_X$ which is a pair of pants, $|E_Y\cap N(F)|\ge 2$.
Let $P_1$ and $P_2$ be two pairs of pants of $E_Y\cap N(F)$ which are adjacent in $N(F)$, and $A$ be a vertical annulus connecting between $P_1$ and $P_2$ such that each component of $\partial A$ is parallel to a component of $\partial P_i$.
Since $E_Y$ is acylindrical, $\partial A$ bounds an annulus $A'$ in a sector of $E_Y$ to which $A$ is parallel, or $A$ is parallel to $A_Y$ relative to $Y$.
In the former case, let $\gamma$ be an arc in $P_1$ which connects two points in $\partial A\cap P_1$ and is not parallel to a subarc in $\partial A\cap P_1$.
Then by a product structure between $P_1$ and $P_2$ and a parallelism between $A$ and $A'$, $\gamma$ extends a compressing disk for $E_Y$.
\begin{figure}[htbp]
	\begin{center}
	\includegraphics[trim=0mm 0mm 0mm 0mm, width=.4\linewidth]{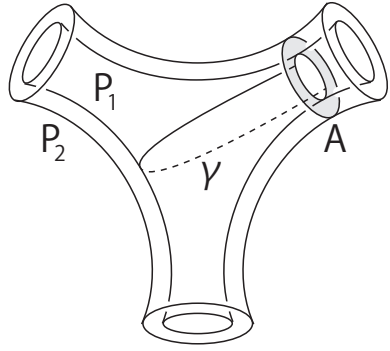}
	\end{center}
	\caption{Configuration of $P_1$, $P_2$, $A$ and $\gamma$}
	\label{pants}
\end{figure}
Therefore we have the latter case.
By repeating this argument for remaining two components of $\partial P_i$, we have the latter cases on them.
This shows that $P_1$ and $P_2$ are extends to mutually parallel pair of pants sectors of $Y$, and 
it contradicts that $Y$ is essential.

As in the proof of Theorem \ref{partial}, by replacing each sector which is not an annulus nor a M\"{o}bius band sector with $C_d\times S^1$, we obtain a multibranched surface $X^*$ from $X$.
The multibranched surface $Y^*$ is simultaneously obtained from $Y$.
Then $S=N(X^*)$ can be regarded as a Seifert fibered 3-manifold with a base orbifold $B$, and $p(Y^*)$ is a graph embedded in $B$, where $p:S\to B$ is the projection.

By Claim \ref{product}, $R$ has a product structure $\partial B\times (0,1]$, where $R$ is the union of regions of $B-N(p(Y^*))$ each of which contains a component of $\partial B$.
However, since $[X]\ne [Y]$, there is a region $R'$ of $B-N(p(Y^*))$ which does not contain any component of $\partial B$, namely, $\partial cl(R')\subset \partial N(p(Y^*))$, where $cl()$ denotes the closure.
The boundary of $p^{-1}(cl(R'))$ consists of annulus or M\"{o}bius band sectors $A_1,\ldots,A_n$ and branches $l_1,\ldots, l_m$ of $Y$.
By the standard form, $cl(R')$ is a neighborhood of a tree embedded in $R'$, that is a disk.
If $\partial cl(R')$ contains $k$ branches $(k\ge 4)$, then there exists an arc $\alpha$ properly embedded in $cl(R')$ such that $\alpha$ is not parallel to a subarc in $p(A_i)$ or $\alpha$ is parallel to $p(A_Y)$ relative to $p(Y^*)$.
\begin{figure}[htbp]
	\begin{center}
	\includegraphics[trim=0mm 0mm 0mm 0mm, width=.6\linewidth]{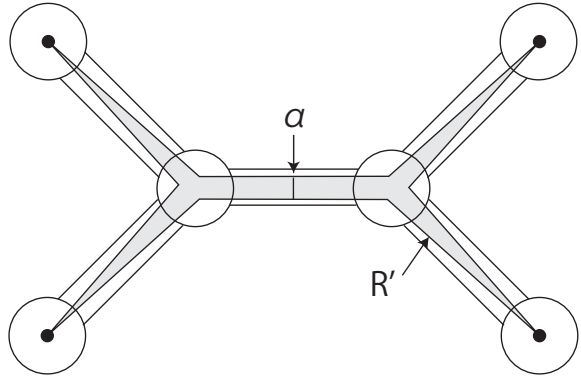}
	\end{center}
	\caption{Configuration of $B$, $R'$, $\alpha$, where $k=4$}
	\label{pants}
\end{figure}
Hence $p^{-1}(\alpha)$ is an essential annulus for $Y$ and it contradicts the supposition that $E_Y$ is acylindrical.
Suppose that $k\le 3$.
If $p^{-1}(R')$ contains a singular fiber, then there exists an arc $\alpha$ properly embedded in $cl(R')$ such that $p^{-1}(\alpha)$ is an essential annulus for $Y$.
Otherwise, one of sectors contained in the boundary of $p^{-1}(cl(R'))$ is excess.
This contradicts the essentiality of $Y$.
\end{proof}





\section{Example}

Let the ambient closed orientable 3-manifold $M$ be the 3-sphere.
We consider four multibranched surfaces $X_1, X_2, X_3, X_4\in \mathcal{X}$ as follows.
For coprime integers $p, q\ge 3$, $B_{X_1}$ is a Hopf link $l_1\cup l_2$ and $E_{X_1}$ consists of a single annulus $A$ such that $f|_{l_1} : a_1 \to l_1$ is $p$-fold and $f|_{l_2} : a_2 \to l_2$ is $q$-fold, where $\partial A=a_1\cup a_2$ and $f:\partial A \to l_1\cup l_2$ is a covering map.
Thus $deg(l_1)=p$ and $deg(l_2)=q$.
Next we take a regular neighborhood $N(l_1)$ and put $l_1'=\partial N(l_1)\cap A$, $A_1=\partial N(l_1)-l_1'$ and $A'=A-N(l_1)$.
The multibranched surface $X_2$ consists of $B_{X_2}=l_1'\cup l_2$ and $E_{X_2}=A'\cup A_1$.
In the same way, we take a regular neighborhood $N(l_2)$ and put $l_2'=\partial N(l_2)\cap A$, $A_2=\partial N(l_2)-l_2'$ and $A''=A-N(l_2)$.
The multibranched surface $X_3$ consists of $B_{X_3}=l_2'\cup l_2$ and $E_{X_3}=A''\cup A_2$.
Finally, put $A'''=A-N(l_1\cup l_2)$.
The multibranched surface $X_4$ consists of $B_{X_4}=l_1'\cup l_2'$ and $E_{X_4}=A'''\cup A_1\cup A_2$.
Then $X_1, X_2, X_3, X_4\in \mathcal{X}$ and we have
\[
[X_1]\prec [X_2], \ [X_1]\prec [X_3], \ [X_2]\prec [X_4], \ [X_3]\prec [X_4], 
\]
and $[X_i]\ne [X_j]$ ($i\ne j$).
Therefore, we have the Hasse diagram as shown in Figure \ref{minimal}.

\begin{figure}[htbp]
	\begin{center}
	\includegraphics[trim=0mm 0mm 0mm 0mm, width=.6\linewidth]{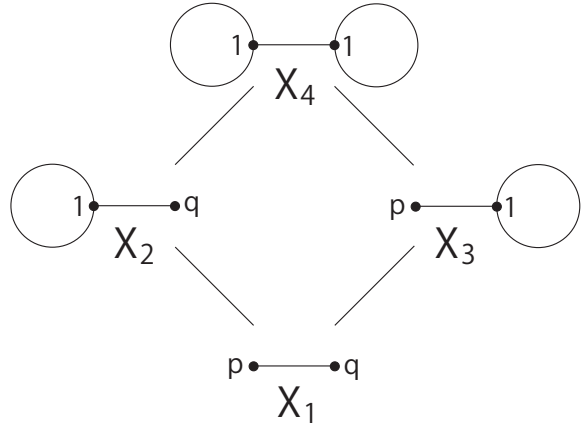}
	\end{center}
	\caption{Hasse diagram for $X_1, X_2, X_3, X_4$}
	\label{minimal}
\end{figure}

\bigskip
\noindent{\bf Acknowledgements.}
I would like to thank to Yuya Koda for pointing out an error on Trisections.

\bibliographystyle{amsplain}

\end{document}